\documentclass[a4paper,11pt,twoside]{amsart}
\usepackage{amsmath}
\usepackage{amsthm}
\usepackage{amsfonts, amssymb}
\usepackage{mathrsfs}
\usepackage[all]{xy}
\usepackage{latexsym}
\usepackage[dvips]{graphicx}
\usepackage{enumerate}

\theoremstyle{plain}
\newtheorem{lemma}{Lemma}[section]
\newtheorem{prop}[lemma]{Proposition}
\newtheorem{theo}[lemma]{Theorem}
\newtheorem{coro}[lemma]{Corollary}
\newtheorem{theoint}[lemma]{Theorem}

\theoremstyle{remark}
\newtheorem{rem}[lemma]{Remark}

\theoremstyle{definition}
\newtheorem{definition}[lemma]{Definition}

\newtheorem{exs}[lemma]{Examples}

\def\hom{\mathrm{Hom}}
\def\id{\mathrm{Id}}
\def\pr{\mathrm{pr}}

\def\incl{\mathrm{inc}}
\def\top{\mathcal{T}\!\!op}

\def\ho{\mathrm{Ho}}
\def\cyl{\mathbb{I}}
\def\cn{\mathrm{C}}

\def\N{\mathbb{N}}

\def\R{\mathbb{R}}

\def\ldash{\textnormal{\textendash}}
\def\sm{\land}  
\def\susp{\Sigma} 
\def\ev{\mathrm{ev}}
\def\tq{\;/\;}

\def\acel{A\textnormal{-}\mathcal{C}el}
\def\cwa{\mathcal{CW}\textnormal{(}A\textnormal{)}}
\def\htcwa{\mathcal{HTCW}\textnormal{(}A\textnormal{)}}
\def\cwagen{\mathcal{CW}\textnormal{(}A\textnormal{)}\textnormal{-}gen}
\def\htcwagen{\mathcal{HTCW}\textnormal{(}A\textnormal{)}\textnormal{-}gen}

\begin{document}

\title{An $A$-based cofibrantly generated model category}

\author{Miguel Ottina}
\email{emottina@uncu.edu.ar}
\thanks{Research partially supported by grant 06/M049 of SeCTyP, UNCuyo.}

\address{Facultad de Ciencias Exactas y Naturales \\
 Universidad Nacional de Cuyo \\ Mendoza, Argentina.}

\begin{abstract}
\noindent
We develop a cofibrantly generated model category structure in the category of topological spaces in which weak equivalences are $A$-weak equivalences and such that the generalized CW($A$)-complexes are cofibrant objects. With this structure the exponential law turns out to be a Quillen adjunction.
\end{abstract}



\maketitle

\section{Introduction} 

CW($A$)-complexes were introduced in \cite{MO1} as topological spaces built up out of simple building blocks or cells of a certain type $A$ (called the \emph{core} of the complex). The $A$-cells are cones of iterated suspensions of the space $A$. The theory of CW($A$)-complexes generalizes the classical theory of CW-complexes of J.H.C. Whitehead \cite{Whi} keeping  his geometric intuition and is related to the theory of $A$-cellular spaces of E.D. Farjoun \cite{Far}. There also exist other generalizations of CW-complexes in the literature, for instance, Baues' generalization of complexes in cofibration categories \cite{Bau} and Minian's categorical approach to cell complexes \cite{Min}.

Much of the theory of CW($A$)-complexes was developed in \cite{OT} where we gave a constructive and a descriptive definition and studied the topological, homological and homotopical properties of these spaces obtaining many interesting results. Among them we mention a generalization of Whitehead's theorem and a change of core theorem \cite{MO1}. We also showed how the homology and homotopy groups of a CW($A$)-complex are related to those of $A$.

Apart from the homology and homotopy groups, other homotopy invariants that naturally enter the picture are the $A$-homotopy groups which are defined in a similar way to the homotopy groups but changing the spheres by iterated suspensions of the space $A$. Clearly, these generalize the homotopy groups with coefficients and, if $A$ is locally compact and Hausdorff, can be seen as the homotopy groups of a certain space of maps. Their importance in this setting is that a continuous map between CW($A$)-complexes is a homotopy equivalence if and only if it induces isomorphisms in the $A$-homotopy groups \cite{MO1}.

A similar result for $A$-cellular spaces had been previously proved by Farjoun in \cite{Far}. Hence, it is important and illustrative to compare both theories. This is done in section \ref{section_CWA_vs_A-cell} of this article where we show that if $A$ is the suspension of a finite CW-complex then the class of $A$-cellular spaces coincides with the class of spaces that have the homotopy type of a generalized CW($A$)-complex. However, the theory of CW($A$)-complexes covers a wider variety of cores $A$ since we do not require the space $A$ to have the homotopy type of a CW-complex. On the other hand, the theory of $A$-cellular spaces, for a fixed CW-complex $A$, deals with a broader range of spaces since the class of spaces which have the homotopy type of generalized CW($A$)-complexes is contained in the class of $A$-cellular spaces, as we shall see.

The main objetive of this article is to define a Quillen's closed model category structure \cite{Qui} in the category of topological spaces, suitable for the theory of CW($A$)-complexes. Hence, the class of weak equivalences should be the class of $A$-weak equivalences, that is, the class of continuous maps which induce isomorphisms in the $A$-homotopy groups.

To this end, in section \ref{section_A_based_model_category} we define notions of $A$-fibrations and $A$-cofibrations which together with the $A$-weak equivalences define a Quillen closed model category structure in the category of topological spaces. Moreover, we prove that this model category is cofibrantly generated. With this definitions the (generalized) CW($A$)-complexes turn out to be cofibrant objects and from the cofibrant replacement functor we obtain a CW($A$)-approximation theorem: given a topological space $X$ there exists a generalized CW($A$)-complex $Z$ together with an $A$-weak equivalence $f:Z\to X$.

Also, applying this model category structure we give a generalization of the $A$-based version of Whitehead's theorem given in \cite{MO1}. In addition, we prove that the adjunction given by the exponential law is a Quillen adjunction. 

This gives a theoretical framework for CW($A$)-complexes and generalizes the classical closed model category structure of pointed and path-connected topological spaces. Thus, we can regard CW($A$)-complexes as a suitable geometrically-flavoured generalization of CW-complexes.

\section{Preliminaries}

In the first part of this section we will recall briefly the main concepts on CW($A$)-complexes. In the second part, we will give some results which will be needed later.

\medskip

{\bf Notations}:
Throughout this article we will work in the category of pointed topological spaces, hence homotopies will always be relative to the base point. $\cyl$, $\cn$ and $\susp$ will denote the reduced cylinder, cone and suspension functors respectively. If $X$ is a (pointed) topological space, $i_0^X:X\to\cyl X$ and $i_1^X:X\to\cyl X$ will denote the standard inclusions into the bottom and the top of the cylinder respectively. Sometimes the superscript $X$ will be dropped and we will simply write $i_0$ and $i_1$ instead of $i_0^X$ and $i_1^X$. Similarly, $i_X$ (or simply $i$) will denote the inclusion $X\hookrightarrow \cn X$.

Also, $S^n$ will denote the $n$-sphere, $D^n$ the $n$-disk and $I$ the unit interval.

In addition, if $Z$ is a space, $Z_+$ will denote the space $Z$ with a disjoint basepoint adjoined.

\medskip

Let $A$ be a fixed pointed topological space.

\begin{definition}
Let $n\in\mathbb{N}$. We say that a (pointed) space \emph{$X$ is obtained from a (pointed) space $B$ by attaching $A$-$n$-cells} if there exists a pushout diagram
\begin{displaymath}
\xymatrix{\displaystyle\bigvee_{\alpha\in J}^{\phantom{J}} \Sigma^{n-1}A \ar[r]^(.64){\underset{\alpha\in J}{+}g_\alpha} \ar[d]_i \ar@{}[rd]|{push} & B \ar[d] \\ \displaystyle\bigvee_{\alpha\in J}^{\phantom{J}}\cn\Sigma^{n-1}A \ar[r]_(.67){\underset{\alpha\in J}{+}f_\alpha} & X}
\end{displaymath}
where $J$ is any set.

The \emph{$A$-cells} are the images of the maps $f_\alpha$. The maps $g_\alpha$ are the \emph{attaching maps} of the cells, and the maps $f_\alpha$ are its \emph{characteristic maps}.
We say that \emph{$X$ is obtained from $B$ by attaching $A$-$0$-cells} if $X=B\lor \displaystyle\bigvee_{\alpha\in J}A$.
\end{definition}

For example, the reduced cone $\cn A$ of $A$ is obtained from $A$ by attaching an $A$-1-cell. In particular, $D^2$ is obtained from $D^1$ by attaching a $D^1$-1-cell. Also, the reduced suspension $\Sigma A$ can be obtained from the singleton $\ast$ by attaching an $A$-1-cell.

Note that attaching an $S^0$-$n$-cell is the same as attaching an $n$-cell in the usual sense (but preserving base points), and that attaching an $S^m$-$n$-cell means attaching an $(m+n)$-cell in the usual sense (preserving base points).

\begin{definition}
A \emph{CW($A$)-complex structure on $X$}, is a sequence of spaces $$\ast=X^{-1}, X^0, X^1, \ldots,X^n,\ldots$$ such that, for $n\in\mathbb{N}_0$, $X^n$ is obtained from $X^{n-1}$ by attaching $A$-$n$-cells, and $X$ is the colimit of the diagram
$$\ast=X^{-1}\rightarrow X^0\rightarrow X^1 \rightarrow \ldots \rightarrow X^n \rightarrow \ldots$$
We say that the space $X$ is a \emph{CW($A$)-complex} if it admits some CW($A$)-complex structure. In this case, the space $A$ will be called the \emph{core} or the \emph{base space} of the structure.
\end{definition}

\begin{exs} \ 
\begin{enumerate}
\item A CW($S^0$)-complex is just a pointed CW-complex with all adjunction maps of cells of positive dimension preserving base points. Clearly, any pointed and path-connected CW-complex has the homotopy type of a CW($S^0$)-complex.

\noindent
Similarly, for all $n\in\N$ a CW($S^n$)-complex is a pointed CW-complex with no cells of dimension less than $n$, apart from the base point. Also, any $n$-connected CW-complex has the homotopy type of a CW($S^n$)-complex.
\item The space $D^n$ admits several different CW($D^1$)-complex structures. For instance, we can take $X^r=D^{r+1}$ for $0 \leq r \leq n-1$ since $\cn D^r$ is homeomorphic to $D^{r+1}$ and $\Sigma^{r-1}D^1$ is homeomorphic to $D^r$ for all $r$. We may also define $$X^0=\ldots = X^{n-2}=\ast \qquad \text{ and } \qquad X^{n-1}=D^n$$ since there is a pushout
\begin{displaymath}
\xymatrix{\Sigma^{n-2}D^1\cong D^{n-1} \ar[r] \ar[d]_i \ar@{}[rd]|{push} & \ast \ar[d] \\ \cn\Sigma^{n-2}D^1 \cong \cn D^{n-1} \ar[r] & \Sigma D^{n-1} \cong D^n}
\end{displaymath}
\end{enumerate}
\end{exs}

\medskip

In \cite{MO1} we analyzed the change from a core $A$ to a core $B$ by means of a map $\alpha:A\to B$ and obtained the following result.

\begin{theoint} \label{changing_cores}
Let $A$ and $B$ be pointed topological spaces with closed base points, let $X$ be a CW($A$)-complex and 
let $\alpha:A\rightarrow B$ and $\beta:B\rightarrow A$ be continuous maps.
\begin{enumerate}
\item[(a)] If $\beta\alpha=\id_A$, then there exist a CW($B$)-complex $Y$ and maps $\varphi:X\rightarrow Y$ 
and $\psi:Y\rightarrow X$ such that $\psi\varphi=\id_X$.
\item[(b)] If $\beta$ is a homotopy equivalence, then there exist a CW($B$)-complex $Y$ and a homotopy 
equivalence $\varphi:X\rightarrow Y$.
\item[(c)] If $\beta\alpha=\id_A$ and $\alpha\beta\simeq\id_B$ then there exist a CW($B$)-complex $Y$ and maps \linebreak $\varphi:X\rightarrow Y$ and $\psi:Y\rightarrow X$ such that $\psi\varphi=\id_X$ and $\varphi\psi\simeq\id_Y$.
\end{enumerate}
\end{theoint}

In particular, when the core $A$ is contractible, all CW($A$)-complexes are also contractible.

\medskip

Let $X$ be a (pointed) topological space. Recall that, for $r\in\mathbb{N}_0$, the sets $\pi^A_r(X)$ are defined by $\pi^A_r(X)=[\Sigma^r A,X]$, this is the homotopy classes of (pointed) maps from $\Sigma^r A$ to $X$. Similarly, for $B\subseteq X$ and $r\in\mathbb{N}$ one defines $\pi^A_r(X,B)=[(\cn\Sigma^{r-1} A,\Sigma^{r-1} A),(X,B)]$.

We say that a space $X$ is $A$-$n$-connected if $\pi^A_r(X)=0$ for $0\leq r\leq n$.

We also say that a continuous map $f:X\to Y$ between (pointed) topological spaces is an \emph{$A$-weak equivalence} if it induces isomorphisms $f_\ast:\pi^A_n(X)\to\pi^A_n(Y)$ for all $n\in\N_0$. 

\begin{rem}
Let $X$ and $Y$ be pointed topological spaces and let $f:X\to Y$ be a continuous map. Let $\mathcal{O}$ denote the forgetful functor from the category of pointed topological spaces to the category of topological spaces. Note that if $\mathcal{O}(f)$ is a weak homotopy equivalence then $f$ is an $S^0$-weak equivalence.

The converse of this implication holds if $X$ is path-connected but does not hold in general as the following example shows. Let $X=S^0$ and let $Y=\{(x,y)\in\R^2\tq (x+2)^2+y^2=1\}\cup \{(1,0)\}$ with the subspace topology with respect to $\R^2$ and with $(1,0)$ as the base point. Let $f:X\to Y$ be defined by $f(x)=(x,0)$. The map $f$ is an $S^0$-weak equivalence but not a weak homotopy equivalence.

Moreover, it is easy to prove that a continuous map $f:X\to Y$ is an $S^0$-weak equivalence if and only if $\mathcal{O}(f)_\ast:\pi_0(X)\to\pi_0(Y)$ is a bijection and the map that $\mathcal{O}(f)$ induces between the path components of $X$ and $Y$ that contain the base points is a weak homotopy equivalence.
\end{rem}

\medskip

In \cite{MO1} we gave the following generalization of Whitehead's theorem. In the next section we will compare this result with a similar one of E.D.Farjoun \cite{Far}.

\begin{theoint}
Let $X$ and $Y$ be $CW(A)$-complexes and let $f:X\to Y$ be a continuous map. Then $f$ is a homotopy equivalence 
if and only if it is an $A$-weak equivalence.
\end{theoint}

We also say that a space $X$ is a \emph{generalized CW($A$)-complex} if it is obtained from the singleton $\ast$ by attaching $A$-cells in countably many steps, allowing cells of any dimension to be attached in any step. More precisely, we have the following definition.

\begin{definition}
Let $X$, $A$ and $B$ be (pointed) topological spaces. We say that $X$ \emph{is obtained from $B$ by attaching $A$-cells} if there exists a pushout
\begin{displaymath}
\xymatrix@C=60pt{\underset{\alpha\in J}{\bigvee}\Sigma^{n_\alpha-1}A \ar[r]^(.6){\underset{\alpha\in J}{+}g_\alpha} \ar[d]_i \ar@{}[rd]|{push} & B \ar[d] \\ (\underset{\alpha\in J_0}{\bigvee}A)\lor(\underset{\alpha\in J}{\bigvee}\cn\Sigma^{n_\alpha-1}A) \ar[r]_(.68){\underset{\alpha\in J_0}{+}f^0_\alpha \  + \ \underset{\alpha\in J}{+}f_\alpha} & X}
\end{displaymath}
where $J_0$ and $J$ are sets and where $n_\alpha\in\mathbb{N}$ for all $\alpha\in J$.

We say that $X$ is a \emph{generalized CW($A$)-complex} if $X$ is the colimit of a diagram
$$\ast=X^0 \rightarrow X^1 \rightarrow X^2 \rightarrow \ldots \rightarrow X^n \rightarrow \ldots$$
where $X^n$ is obtained from $X^{n-1}$ by attaching $A$-cells.

We call $X^n$ the \emph{$n$-th layer} of $X$.

We also say that a (pointed) topological pair $(X,B)$ is a \emph{relative CW($A$)-complex} if $X$ is the colimit of a diagram
$$B=X^{-1} \rightarrow X^0 \rightarrow X^1 \rightarrow X^2 \rightarrow \ldots \rightarrow X^n \rightarrow \ldots$$
where $X^n$ is obtained from $X^{n-1}$ by attaching $A$-$n$-cells.

Finally, we say that a (pointed) topological pair $(X,B)$ is a \emph{generalized relative CW($A$)-complex} if $X$ is the colimit of a diagram
$$B=X^0 \rightarrow X^1 \rightarrow X^2 \rightarrow \ldots \rightarrow X^n \rightarrow \ldots$$
where $X^n$ is obtained from $X^{n-1}$ by attaching $A$-cells.
\end{definition}

A generalized CW($S^0$)-complex will be called a generalized CW-complex. It is easy to prove that a generalized CW-complex has the homotopy type of a CW-complex. Moreover, we prove in \cite{MO1} that if $A$ is any CW-complex then a generalized CW($A$)-complex has the homotopy type of a CW-complex. However, a generalized CW($A$)-complex does not necessarily have the homotopy type of a CW($A$)-complex (see \cite{MO2}, example 3.10).

\medskip

Now we will develop some results which will be needed later. We start by studying under which conditions the reduced cylinder of a CW($A$)-complex is again a CW($A$)-complex.

\begin{lemma}
Let $\nu:S^1\to S^1 \lor S^1$ be the usual map inducing the comultiplication in $S^1$. Then there is a pushout
\begin{displaymath}
\xymatrix{ S^1 \ar[r]^(.4)\nu \ar[d] \ar@{}[rd]|{push} & S^1\lor S^1 \ar[d]^i \\ \cn S^1 \ar[r] & \cyl S^1}
\end{displaymath}
where $i(S^1\lor S^1)=i_0(S^1)\lor i_1(S^1)$.
\end{lemma}

\begin{proof}
Note that the pushout of the diagram
\begin{displaymath}
\xymatrix{ S^1 \ar[r]^(.4)\nu \ar[d]  & S^1\lor S^1  \\ \cn S^1  & }
\end{displaymath}
is $D^2/\{(-1,0),(1,0)\}$. There are homeomorphisms
$$\cyl S^1 = S^1 \times I / (\{(1,0)\}\times I) \cong (I \times I) / (\{0,1\}\times I) \cong D^2/\{(-1,0),(1,0)\}$$
and hence, the result follows.
\end{proof}

\begin{prop} \label{IX_is_CW(A)}
Let $A'$ be a locally compact and Hausdorff space and let $A=\Sigma A'$. Let $X$ be a CW($A$)-complex. Then the reduced cylinder $\cyl X$ is a CW($A$)-complex. Moreover, $i_0(X)$ and $i_1(X)$ are CW($A$)-subcomplexes of $\cyl X$.
\end{prop}

\begin{proof}
For $n\in\mathbb{N}$ let $J_n$ be an index set for the $A$-$n$-cells of $X$. We proceed by induction in the $A$-skeletons of $X$. For the initial case we have that $\displaystyle X^0=\bigvee_{\alpha\in J_0} A$. Then $\displaystyle \cyl X^0=\bigvee_{\alpha\in J_0} \cyl A$. But $\cyl A$ is a CW($A$)-complex since applying $\ldash\land A'$ to the pushout of the previous lemma gives a pushout
\begin{displaymath}
\xymatrix{ \Sigma A'=A \ar[r] \ar[d] \ar@{}[rd]|{push} & \Sigma A' \lor \Sigma A'=A \lor A \ar[d] \\ \cn \Sigma A' = \cn A \ar[r] & \cyl\Sigma A' = \cyl A}
\end{displaymath}
since $\cn S^1 \sm A' \cong I \sm S^1 \sm A' \cong I \sm \susp A' \cong \cn \susp A'$ and  $\cyl S^1 \sm A' \cong I_+ \sm S^1 \sm A' \cong I_+ \sm \susp A' \cong \cyl \susp A'$. Moreover, $i_0(A)$ and $i_1(A)$ are CW($A$)-subcomplexes of $\cyl A$. Hence, $\cyl X^{0}$ is a CW($A$)-complex and $i_0(X^{0})$ and $i_1(X^{0})$ are CW($A$)-subcomplexes of $\cyl X^{0}$.

Now suppose that $\cyl X^{n-1}$ is a CW($A$)-complex and that $i_0(X^{n-1})$ and $i_1(X^{n-1})$ are CW($A$)-subcomplexes of $\cyl X^{n-1}$. We consider
\begin{displaymath}
\xymatrix{ \underset{\alpha \in J_{n-1} \sqcup J_{n-1}}{\bigvee} \!\!\!\!\!\!\!\!\!\!\Sigma^{n-1} A \ar[r]^(.6){g} \ar[d] \ar@{}[rd]|{push} & \cyl X^{n-1} \ar[d]^{\incl} \\ \underset{\alpha \in J_{n-1} \sqcup J_{n-1}}{\bigvee} \!\!\!\!\!\!\!\!\!\!\cn\Sigma^{n-1} A \ar[r]_(.65){\underset{\alpha}{+}f^n_\alpha} & Y_n}
\qquad
\xymatrix{ \underset{\alpha \in J_{n}}{\bigvee} \Sigma^{n} A \ar[r]^(.6){\underset{\alpha}{+}G^n_\alpha} \ar[d] \ar@{}[rd]|{push} & Y_n \ar[d] \\ \underset{\alpha \in J_{n}}{\bigvee} \cn\Sigma^{n} A \ar[r] & Z_n}
\end{displaymath}
with $g=\left(\underset{\alpha \in J_{n-1}}{+}\!\!\!\! i_0 g^n_\alpha\right) + \left(\underset{\alpha \in J_{n-1}}{+}\!\!\!\! i_1 g^n_\alpha\right)$ where $(g^n_\alpha)_{\alpha\in J_{n-1}}$ are the adjunction maps of the $A$-$n$-cells of $X$, and where the maps $G_\alpha$, $\alpha\in J_n$, are defined as the composition
\begin{displaymath}
\xymatrix@C=50pt{\Sigma(\Sigma^{n-1}A) \ar[r] & \cn(\Sigma^{n-1}A) \underset{\Sigma^{n-1}A}{\bigcup} \cyl(\Sigma^{n-1}A) \underset{\Sigma^{n-1}A}{\bigcup} \cn(\Sigma^{n-1}A) \ar[r]^(.8){f^n_\alpha \cup F_\alpha \cup f^n_\alpha} & Y_n}
\end{displaymath}
where the first map is the standard homeomorphism and where $F_\alpha$ is the composition
\begin{displaymath}
\xymatrix{\cyl(\Sigma^{n-1}A) \ar[r]^(.55){\cyl g^n_\alpha} & \cyl X^{n-1} \ar[r]^(.55){\incl} & Y_n}
\end{displaymath}
We wish to prove that $Z_n$ is homeomorphic to $\cyl X^n$. Note that 
$$Y_n=X^n \underset{X^{n-1}}{\bigcup} \cyl X^{n-1} \underset{X^{n-1}}{\bigcup} X^n.$$

We have that
\begin{displaymath}
\xymatrix{ \underset{\alpha \in J_{n-1} \sqcup J_{n-1}}{\bigvee} \Sigma^{n-1} A \ar[r]^(.45){g} \ar[d] \ar@{}[rd]|{push} & X^{n-1}\times\{0\}\lor X^{n-1}\times\{1\} \ar[r]^(.7){\incl} \ar[d] \ar@{}[rd]|{push} & \cyl X^{n-1} \ar[d] \\ \underset{\alpha \in J_{n-1} \sqcup J_{n-1}}{\bigvee} \cn\Sigma^{n-1} A \ar[r]_(.65){\underset{\alpha}{+}f^n_\alpha} & W_n \ar[r] & Y_n}
\end{displaymath}
and clearly $W_n=X^n\times\{0\}\lor X^n\times\{1\}$.

Now, the homeomorphism $\Sigma(\Sigma^{n-1}A) \longrightarrow \cn(\Sigma^{n-1}A) \underset{\Sigma^{n-1}A}{\bigcup} \cyl (\Sigma^{n-1}A) \underset{\Sigma^{n-1}A}{\bigcup} \cn(\Sigma^{n-1}A)$ extends to a homeomorphism $\cn(\Sigma^{n}A)\longrightarrow \cyl \cn(\Sigma^{n-1}A)$. Indeed, this follows applying $\ldash\land A$ to the homeomorphism of topological pairs $\psi:(D^{n+1},S^n)\to(\cyl D^n, D^{n} \underset{S^{n-1}}{\cup} \cyl S^{n-1} \underset{S^{n-1}}{\cup} D^{n})$ 

Then, we have
\begin{displaymath}
\xymatrix@C=30pt{\underset{\alpha \in J_{n}}{\bigvee}(\cn\Sigma^{n-1}A \underset{\Sigma^{n-1}A}{\bigcup} \cyl\Sigma^{n-1}A \underset{\Sigma^{n-1}A}{\bigcup} \cn\Sigma^{n-1}A) \ar[r] \ar[d] \ar@{}[rd]|{push} & \underset{\alpha \in J_{n}}{\bigvee} \Sigma^{n} A \ar[r]^{\underset{\alpha}{+}G^n_\alpha} \ar[d] \ar@{}[rd]|{push} & Y_n \ar[d] \\\underset{\alpha \in J_{n}}{\bigvee} \cyl\cn\Sigma^{n-1}A \ar[r] & \underset{\alpha \in J_{n}}{\bigvee} \cn\Sigma^{n} A \ar[r] & Z_n}
\end{displaymath}
Note that the first square is a pushout since it commutes and its two horizontal arrows are homeomorphisms.

Note also that the top horizontal composition is $\underset{\alpha}{+}(f^n_\alpha \cup F_\alpha \cup f^n_\alpha)$ and that $Z_n=\cyl X^n$ since $F_\alpha=\incl \circ \cyl g^n_\alpha$. The result follows.
\end{proof}

The following two lemmas will be used in section \ref{section_A_based_model_category}.

\begin{lemma} \label{lemma_homeo_pairs_1}
Let $B$ be a (pointed) topological space. Then there exist homeomorphisms of pointed topological pairs:
\begin{enumerate}
\item[(a)] $(\cyl\cn B, i_0(\cn B)) \cong (\cyl\cn B, i_0(\cn B )\cup \cyl B)$
\item[(b)] $(\cyl\cn B, i_0(\cn B)\cup \cyl (i(B)) \cup i_1(\cn B)  )\cong (\cn \Sigma B, \Sigma B)$
\item[(c)] $(\cyl\,\cyl B,i_0(\cyl B))\cong (\cyl\,\cyl B,\cyl(i_0(B)) \cup i_0(\cyl B)))$
\end{enumerate}
\end{lemma}

\begin{proof}
Recall that if $Z$ is a topological space, $Z_+$ denotes the space $Z$ with a disjoint basepoint adjoined. Also, in what follows $0$ will be the base point of the spaces $I$ and $\{0,1\}\subseteq I$ and the point $(1,0)\in\R^2$ will be the base point of the disk $D^2$.

Let $T=\{(x,y)\in S^1 \tq y\leq 0 \}$.

(a) Note that the topological pairs $(I_+\wedge I, \{0\}_+ \wedge I)$ and $(I_+\wedge I, (\{0\}_+ \wedge I) \cup (I_+ \wedge \{0,1\}))$ are homeomorphic since both of them are homeomorphic to $(D^2,T)$. Hence, item (a) follows taking smash product with $B$.

(b) Note that the topological pairs $(I_+\wedge I, (\{0\}_+ \wedge I) \cup (I_+ \wedge \{0,1\}) \cup (\{1\}_+ \wedge I))$ and $(I\wedge S^1, \{0,1\} \wedge S^1)$ are homeomorphic since both of them are homeomorphic to $(D^2,S^1)$. Hence, item (b) follows taking smash product with $B$.

(c) Note that the topological pairs $(I_+\wedge I_+, \{0\}_+ \wedge I_+)$ and $ (I_+\wedge I_+, (\{0\}_+ \wedge I_+) \cup (I_+ \wedge \{0\}_+))$ are homeomorphic since both of them are homeomorphic to $(D^2_+,T_+)$. Hence, item (c) follows taking smash product with $B$.
\end{proof}

\begin{lemma} \label{lemma_homeo_pairs_2}
Let $B$ be a topological space. We define in $\cyl\cn B$ an equivalence relation as follows. If $x,x'\in\cn B$ and $t,t'\in I$ we say that $(x,t)\sim(x',t')$ if and only if $x\in B$ and $x=x'$. Let $q:\cyl\cn B \to \cyl\cn B/\sim$ be the quotient map. Let $\susp_+B=\{[b,t]\in\susp B \;/\; t\geq \frac{1}{2}\}$ and $\susp_-B=\{[b,t]\in\susp B \;/\; t\leq \frac{1}{2}\}$. Then there exists a homeomorphism $\varphi:\cyl\cn B/\!\sim\ \to \cn\susp B$ such that $\varphi(q(i_0(\cn B)))=\susp_-B$ and $\varphi(q(i_1(\cn B)))=\susp_+B$.
\end{lemma}

\begin{proof} Let $\sim_0$ be the equivalence relation in $\cyl\cn S^0$ defined as the one above, that is, if $x,x'\in\cn S^0$ and $t,t'\in I$, then $(x,t)\sim_0(x',t')$ if and only if $x\in S^0$ and $x=x'$. There are homeomorphisms
\begin{displaymath}
\begin{array}{rcl}
\cyl\cn B/\!\sim & \cong & I_+\sm I \sm B /\!\sim 
\ \cong I_+\sm I \sm S^0 \sm B /\!\sim 
\ \cong \cyl\cn S^0 \sm B /\!\sim 
\ \cong \cyl\cn S^0 /\!\sim_0 \sm B  \ \cong  \\ 
& \cong & \cyl D^1 /\!\sim_0 \sm B 
\ \cong D^2 \sm B 
\ \cong \cn S^1 \sm B 
\ \cong I \sm S^1 \sm B 
\ \cong \cn\susp B.
\end{array}
\end{displaymath}
The result follows.
\end{proof}

\section{Comparison between CW($A$)-complexes and Dror Farjoun's $A$-cellular spaces} \label{section_CWA_vs_A-cell}

In this section we will compare CW($A$)-complexes \cite{OT} with $A$-cellular spaces \cite{Far}. We begin by recalling Dror Farjoun's definition of $A$-cellular spaces.

\begin{definition}
Let $A$ be a (pointed) topological space having the homotopy type of a CW-complex. The \emph{class of $A$-cellular spaces} is the smallest class of topological spaces that contains the space $A$ and is closed under homotopy equivalences and pointed homotopy colimits.
\end{definition}

Note that in Dror Farjoun's setting the space $A$ is required to have the homotopy type of a CW-complex, while we allow $A$ to be any space. Hence, our results are more general to a certain extent. As an illustrative example, let us compare our version of the $A$-Whitehead theorem (\cite{MO1}, theorem 5.10) with the one given in \cite{Far} (theorem E.1 of chapter 2). If $A$ is a suspension of a finite CW-complex then both versions coincide (cf. proposition \ref{prop_htcwagen=acel}). But our version covers a wider variety of cores $A$, while Dror Farjoun's one, for a fixed CW-complex $A$, deals with a broader range of spaces since the class of CW($A$)-complexes is strictly contained in the class of $A$-cellular spaces as we shall see.

Until the end of this section, $A$ will denote a fixed pointed topological space having the homotopy type of a CW-complex.

Before going on, we fix the following notation.
\begin{itemize}
\item $A$-$\mathcal{C}el$ will denote the class of $A$-cellular spaces.
\item $\mathcal{CW}$($A$) will denote the class of CW($A$)-complexes.
\item $\mathcal{HTCW}$($A$) will denote the class of spaces having the homotopy type of a CW($A$)-complex.
\item $\mathcal{CW}$($A$)-$gen$ will denote the class of generalized CW($A$)-complexes.
\item $\mathcal{HTCW}$($A$)-$gen$ will denote the class of spaces having the homotopy type of a generalized CW($A$)-complex.
\end{itemize}

We aim to compare these five classes of topological spaces. It is clear that
$$\cwa\subseteq \cwagen \subseteq\htcwagen$$
and
$$\cwa \subseteq \htcwa\subseteq \htcwagen.$$

Also, since the constructions of CW($A$)-complexes and generalized CW($A$)-complexes can be rephrased in terms of homotopy colimits we get that $$\htcwagen\subseteq \acel.$$

Now, let $CW_A$ denote the $A$-cellular approximation functor described in \cite{Far}. From section 2.E of that work, we conclude that if $A$ is the suspension of a finite CW-complex and $Z$ is a topological space which has the homotopy type of a CW-complex then $CW_A(Z)$ can be obtained by an inductive process of attaching $A$-cells. Therefore, $CW_A(Z)$ is a generalized CW($A$)-complex. Then, if $X$ is an $A$-cellular space we obtain that $X$ is homotopy equivalent to $CW_A(X)$ and hence $X$ has the homotopy type of a generalized CW($A$)-complex. Summing up, we have proved the following.

\begin{prop} \label{prop_htcwagen=acel}
If $A$ is the suspension of a finite CW-complex then $$\htcwagen = \acel.$$
\end{prop}

However, the other inclusions are strict in general. Indeed, in example 3.10 of \cite{MO2} we exhibit a generalized CW($A$)-complex which does not have the homotopy type of a CW($A$)-complex. This shows that $\cwa\subsetneq \cwagen$ and $\htcwa\subsetneq \htcwagen$. On the other hand, the inclusions $\cwa \subseteq \htcwa$ and $\cwagen \subseteq\htcwagen$ are easily seen to be strict.

\section{An $A$-based cofibrantly generated model category} \label{section_A_based_model_category}

In this section we will define appropriate notions of $A$-fibrations and $A$-cofibrations and prove that together with the $A$-weak equivalences they define a closed model category structure in the category of pointed topological spaces. Moreover, we will prove that this model category is cofibrantly generated.

Recall that if $\mathcal{C}$ is a category and $\mathcal{I}$ is a set of maps in $\mathcal{C}$ then the class of $\mathcal{I}$-injectives is the class of maps that have the right lifting property (RLP) with respect to every map in $\mathcal{I}$ and is denoted by $\mathcal{I}$-inj. Also, the class of $\mathcal{I}$-cofibrations is the class of maps that have the left lifting property (LLP) with respect to every map in $\mathcal{I}$-inj and is denoted by $\mathcal{I}$-cof.

\medskip

For the rest of this section, let $A$ be a fixed pointed topological space.

Let 
$$\mathcal{I}_A=\{\Sigma^{n-1}A\hookrightarrow\cn\Sigma^{n-1}A \tq n\in\N\}\cup\{i_0:\cyl^{n-1}A\rightarrow\cyl^{n}A  \tq n\in\N\} \cup \{\ast\to A\}$$
and let
$$\mathcal{J}_A= \{i_0:\cyl^{n-1}A\rightarrow\cyl^{n}A  \tq n\in\N\} \cup \{i_0:\cn\Sigma^{n-1}A\rightarrow\cyl\cn\Sigma^{n-1}A  \tq n\in\N\}$$
where $\cyl^0$ and $\Sigma^0$ denote the identity functors.

The set $\mathcal{I}_A$ will be the set of generating cofibrations and the set $\mathcal{J}_A$ will be the set of generating trivial cofibrations of our model category, as we shall see.

\begin{definition}
Let $E$ and $B$ be (pointed) topological spaces and let $p:E\to B$ be a continuous map. We say that $p$ is an \emph{$A$-fibration} if it has the homotopy lifting property (HLP) with respect to the spaces $A$ and $\cn\Sigma^{n-1}A$, $n\in\N$, and $\cyl^nA$, $n\in\N$. 

The class of $A$-fibrations will be denoted by $\textnormal{Fib}$.
\end{definition}

Note that $\textnormal{Fib}=\mathcal{J}_A\textnormal{-inj}$.

Let $\textnormal{WE}$ be the class of $A$-weak equivalences and let $\textnormal{Cof}$ be the class of maps which have the left lifting property with respect to the class $\textnormal{Fib}\cap\textnormal{WE}$. The maps in the class $\textnormal{Cof}$ will be called $A$-cofibrations. If $X$ is a topological space such that the map $\ast\to X$ is an $A$-cofibration, we will say that $X$ is an $A$-cofibrant space.

The following proposition shows that $S^0$-fibrations are just Serre fibrations.

\begin{prop} \label{prop_S0_fib}
Let $X$ and $Y$ be pointed topological spaces and let $p:X\to Y$ be a continuous map. Let $\mathcal{O}$ denote the forgetful functor from the category of pointed topological spaces to the category of topological spaces. Then the map $p$ is an $S^0$-fibration if and only if $\mathcal{O}(p)$ is a Serre fibration.
\end{prop}

\begin{proof}
Suppose first that $p$ is an $S^0$-fibration. By definition, $p$ has the HLP with respect to $\cyl^nS^0$ for all $n\in\N_0$. Note that, for $n\in \N$, $\cyl^nS^0$ is equal to $I^n_+$ which is homeomorphic to $D^n_+$ and $\cyl^0S^0=S^0$ is homeomorphic to $D^0_+$. It follows that $\mathcal{O}(p)$ has the (unbased) homotopy lifting property with respect to $D^n$ for all $n\in\N_0$. Hence, $\mathcal{O}(p)$ is a Serre fibration.

Conversely, if the map $\mathcal{O}(p)$ is a Serre fibration then it has the homotopy lifting property with respect to $D^n$ for all $n\in\N_0$. Since $\cyl^nS^0$ is homeomorphic to $D^n_+$ for all $n\in\N_0$ it follows that the map $p$ has the HLP with respect to $\cyl^nS^0$ for all $n\in\N_0$. It remains to prove that $p$ has the HLP with respect to $\cn\Sigma^{n-1}S^0$ for all $n\in\N$.

Since $\cn\Sigma^{n-1}S^0$ is homeomorphic to $D^n$ which is homeomorphic to $\cyl^{n-1}I$, it suffices to prove that $p$ has the RLP with respect to the inclusion $i_0:\cyl^{n-1}I \to \cyl^{n}I$ for all $n\in\N$. 

Let $n\in\N$ and let $f:\cyl^{n-1}I\to X$ and $H:\cyl^{n}I\to Y$ be continuous maps such that $Hi_0=pf$. For $j\in\N$ let $q_{j}:I^j\to\cyl^{j-1}I\cong I^j / (\{0\}\times I^{j-1})$ be the quotient map. Let $x_0$ and $y_0$ be the base points of $X$ and $Y$ respectively and let $F:(I^n\times\{0\})\cup(\{0\}\times I^n)\to X$ be defined by $F(z,0)=f(q_n(z))$ and $F(0,z)=x_0$ for $z\in I^n$. Clearly, the map $F$ is well-defined  and continuous.

There is a commutative diagram of unpointed spaces and maps
\begin{displaymath}
\xymatrix{(I^n\times\{0\})\cup(\{0\}\times I^n) \ar[r]^(.68){F} \ar[d]_{\incl} & \mathcal{O}(X) \ar[d]^{\mathcal{O}(p)} \\ I^{n+1}
  \ar[r]_(.5){Hq_{n+1}} & \mathcal{O}(Y)}
\end{displaymath}
Now, since the topological pairs $(I^{n+1},I^n\times\{0\})$ and $(I^{n+1},(I^n\times\{0\})\cup(\{0\}\times I^n))$ are homeomorphic and $\mathcal{O}(p)$ is a Serre fibration, there exists a continuous map $G:I^{n+1}\to \mathcal{O}(X)$ such that $G\,\incl=F$ and $\mathcal{O}(p)G=Hq_{n+1}$. Since $G(0,z)=F(0,z)=x_0$ for all $z\in I^n$, there exists a continuous map $\overline{G}:\cyl^n I \to \mathcal{O}(X)$ such that $\overline{G}q_{n+1}=G$. The map $\overline{G}$ is the desired lift.
\end{proof}

\begin{prop}
Let $E$ and $B$ be topological spaces and let $p:E\to B$ be an $A$-fibration. Then $p$ has the homotopy lifting property with respect to the class of generalized CW($A$)-complexes.
\end{prop}

The proof follows by standard arguments applying the homeomorphism given in item (a) of \ref{lemma_homeo_pairs_1}.

\medskip

Note that if the core $A$ is the suspension of a locally compact and Hausdorff space then a map $p$ is an $A$-fibration if and only if it has the homotopy lifting property with respect to the spaces $A$ and $\cn\Sigma^{n-1}A$, $n\in\N$, since in this case, by \ref{IX_is_CW(A)}, $\cyl^n A$ is a CW($A$)-complex for all $n\in\N$.

\bigskip

We will prove now some results which will be needed for our purposes.

\begin{prop} \label{propfibyed}
Let $E$, $B$ topological spaces and let $p:E\rightarrow B$ be a continuous map. If $A$ is an $H$-cogroup, the following are equivalent
\begin{enumerate}
\item[(a)] $p\in \textnormal{Fib}\cap\textnormal{WE}$.
\item[(b)] The map $p$ has the RLP with respect to the inclusion maps $\Sigma^{n-1}A\hookrightarrow \cn\Sigma^{n-1}A$ $\forall n\in\mathbb{N}$ and with respect to the maps $\ast \to A$ and $i_0: \cyl^{n-1} A\to \cyl^n A$ $\forall n\in\mathbb{N}$.
\end{enumerate}
Moreover, any of these implies the following
\begin{enumerate}
\item[(c)] The map $p$ has RLP with respect to the inclusion $B\hookrightarrow X$ for all generalized relative CW($A$)-complexes $(X,B)$.
\end{enumerate}
\end{prop}

\begin{proof}
(a) $\Rightarrow$ (b) We consider first a commutative diagram
\begin{displaymath}
\xymatrix{\Sigma^{n-1}A \ar[r]^(0.6)f \ar[d]_i & E
  \ar[d]^p \\ \cn\Sigma^{n-1}A \ar[r]_(0.65)h & B }
\end{displaymath}
where $n\in\N$.

Since $p$ is an $A$-weak equivalence, from lemma 5.6 of \cite{MO1} we get that there exists a continuous map 
$h':\cn\Sigma^{n-1}A\rightarrow E$ such that $h'|_{\Sigma^{n-1}A}=f$
and $ph'\simeq h$ rel $\Sigma^{n-1}A$.
Let $H:\cyl\cn\Sigma^{n-1}A\rightarrow B$ be a homotopy from $ph'$ to $h$ relative to $\Sigma^{n-1}A$. There is a commutative diagram
\begin{displaymath}
\xymatrix{i_0(\cn\Sigma^{n-1} A)\cup \cyl(\Sigma^{n-1}A) \ar[r]^(.78){h'\cup \cyl f}
  \ar[d]_\incl & E \ar[d]^p \\ \cyl\cn\Sigma^{n-1} A \ar[r]_(.6){H} & B }
\end{displaymath}
Since $p$ is an $A$-fibration, from item (a) of \ref{lemma_homeo_pairs_1} we conclude that there exists a map $H':\cyl\cn\Sigma^{n-1} A \rightarrow E$ such that $pH'=H$ and $H'\incl=h'\cup \cyl f$. The map $H'i_1:\cn\Sigma^{n-1}A\rightarrow E$ is the desired lift.

Consider now a commutative diagram
\begin{displaymath}
\xymatrix{\ast \ar[r] \ar[d] & E
  \ar[d]^p \\ A \ar[r]_g & B }
\end{displaymath}
Since $p$ is an $A$-weak equivalence there exists a continuous map $\overline{g}:A\to E$ such that $p\overline{g} \simeq g$. Let $H:\cyl A\to B$ be a homotopy between $p\overline{g}$ and $g$. There is a commutative diagram
\begin{displaymath}
\xymatrix{A \ar[r]^{\overline{g}} \ar[d]_{i_0} & E
  \ar[d]^p \\ \cyl A \ar[r]_(.52){H} & B }
\end{displaymath}
Since $p$ is an $A$-fibration there exists a continuous map $\overline{H}:\cyl A\to E$ such that $\overline{H}i_0=\overline{g}$ and $p\overline{H}=H$. Hence, $p\overline{H}i_1=Hi_1=g$. Thus $\overline{H}i_1$ is the desired lift.

Finally, the RLP with respect to the maps $i_0: \cyl^{n-1} A\to \cyl^n A$ $\forall n\in\mathbb{N}$ follows from the definition of $A$-fibration.

(b) $\Rightarrow$ (a) We will prove first that $p$ is an $A$-fibration. Let $b_0$ be the base point of $B$, let $n\in\N$ and let $f:\cn\Sigma^{n-1}A\to E$ and $H:I\cn\Sigma^{n-1}A\to B$ be continuous maps such that $Hi_0=pf$.
\begin{displaymath}
\xymatrix{\cn\Sigma^{n-1}A \ar[r]^(.63){f} \ar[d]_{i_0} & E \ar[d]^p \\ \cyl\cn\Sigma^{n-1}A
  \ar[r]_(.65){H} & B}
\end{displaymath}

Let $G:\cyl(\cyl\Sigma^{n-1}A) \to B$ be defined by
\begin{displaymath}
G(a,s,t)=\left\{ \begin{array}{ll}
H([[a,\frac{s}{1-t}],2t]) & \textnormal{if $a\in\Sigma^{n-1}A$, $0\leq s \leq 1$ and $0\leq t \leq \min\{\frac{1}{2},1-s\}$} \\ 
H([[a,\frac{s}{t}],2(1-t)]) & \textnormal{if $a\in\Sigma^{n-1}A$, $0\leq s \leq 1$ and $\max\{\frac{1}{2},s\}\leq t \leq 1$} \\ 
H([[a,1],2(1-s)]) & \textnormal{if $a\in\Sigma^{n-1}A$, $0\leq s \leq 1$ and $1-s \leq t \leq s $} 
\end{array} \right.
\end{displaymath}
It is easy to check that $G$ is well-defined and continuous. Moreover, since $G(a,0,t)=b_0$ for all $a\in\Sigma^{n-1}A$ and $t\in [0,1]$, then there exists a map $\overline{G}:\cyl\cn\Sigma^{n-1}A \to B$ such that $\overline{G}q=G$, where $q:\cyl(\cyl\Sigma^{n-1}A) \to \cyl\cn\Sigma^{n-1}A$ is the quotient map.

Let $D=i_0(\cn \Sigma^{n-1}A) \cup \cyl i(\Sigma^{n-1}A) \cup i_1(\cn \Sigma^{n-1}A) \subseteq \cyl \cn \Sigma^{n-1}A$. Let $F:D\to B$ be defined by
\begin{displaymath}
F(a,s,t)=\left\{ \begin{array}{ll}
f([a,s]) & \textnormal{if $a\in\Sigma^{n-1}A$, $0\leq s \leq 1$ and $t\in \{0,1\}$ } \\ 
f([a,1]) & \textnormal{if $a\in\Sigma^{n-1}A$, $s=1$ and $0\leq t \leq 1$ }
\end{array} \right.
\end{displaymath}

It is easy to check that $F$ is well-defined and continuous. Moreover, there is a commutative diagram
\begin{displaymath}
\xymatrix{D \ar[r]^(.5){F} \ar[d]_{\incl} & E \ar[d]^p \\ \cyl\cn\Sigma^{n-1}A
  \ar[r]_(.65){\overline{G}} & B}
\end{displaymath}

By hypothesis and applying item (b) of \ref{lemma_homeo_pairs_1} we obtain that there exists a map $K:\cyl\cn\Sigma^{n-1}A \to E$ such that $K\incl=F$ and $pK=\overline{G}$.

Let $k:\cyl\cn\Sigma^{n-1}A \to E$ be defined by $$k([[a,s],t])=K([[a,s(1-\tfrac{t}{2})],\tfrac{t}{2}]).$$

It is easy to check that $ki_0=f$ and $pk=H$.

The homotopy lifting properties with respect to $A$ and $\cyl^nA$, $n\in\N$, follow from the hypothesis.

Thus, $p$ is an $A$-fibration.

We will prove now that $p$ is an $A$-weak equivalence. We will prove first that $p$ induces isomorphisms $p_\ast:\pi^A_n(E)\to\pi^A_n(B)$ for $n\geq 1$.

Let $n\in\N$ and let $f,g:\Sigma^{n}A\rightarrow E$ be continuous maps such that $pf\simeq pg$. Let 
$H:\cyl\Sigma^{n}A\rightarrow B$ be the homotopy.

Let $q:\cn\Sigma^{n-1}A\rightarrow \cn\Sigma^{n-1}A/\Sigma^{n-1}A\cong \Sigma^n A$ be the quotient map.
Hence $fq,gq:\cn\Sigma^{n-1} A\rightarrow E$ and $$H\circ \cyl q:\cyl \cn\Sigma^{n-1} A \rightarrow B$$
is a homotopy between $pfq$ and $pgq$ relative to $\Sigma^{n-1}A$.

We consider in $\cyl\cn\Sigma^{n-1} A$ the following equivalence relation: 
$$(x,t)\sim(x',t') \Leftrightarrow x\in \Sigma^{n-1} A \textnormal{ and } x=x'$$

By \ref{lemma_homeo_pairs_2} there exists a homeomorphism $\cyl\cn\susp^{n-1} A /\!\sim\ \cong \cn\Sigma^{n}A$ which takes the classes of the elements of the bottom and the top of the cylinder to $\Sigma_-\Sigma^{n-1}A$ and $\Sigma_+\Sigma^{n-1}A$ respectively. Let $q':\cyl\cn\Sigma^{n-1} A \to \cyl\cn\Sigma^{n-1}A/\sim\ \cong \cn\Sigma^{n}A$ be the quotient map. Since the homotopy $H\circ \cyl q$ is relative to $\Sigma^{n-1}A$, there exists a continuous map $H':\cn\Sigma^{n}A\rightarrow B$ such that $H'q'=H\circ \cyl q$.

Note that $H'|_{\susp^n A}$ equals $pfq$ in the southern hemisphere $\Sigma_-\Sigma^{n-1}A\cong \cn\Sigma^{n-1}A$ and $pgq$ in the northern hemisphere $\Sigma_+\Sigma^{n-1}A\cong \cn\Sigma^{n-1}A$. Hence, there is a commutative diagram
\begin{displaymath}
\xymatrix@C=30pt{\Sigma^n A=\Sigma^n_- A\cup \Sigma^n_+ A \ar[r]^(.72){fq\cup gq}
  \ar[d]_\incl & E \ar[d]^p \\ \cn\Sigma^{n}A \ar[r]_(.58){H'} & B}
\end{displaymath}

By hypothesis there exists a continuous map $K:\cn\Sigma^{n}A\rightarrow E$ such that $pK=H'$ and $K|_{\Sigma^nA}=fq\cup gq$.  Then
$Kq':\cyl\cn\Sigma^{n-1}A\rightarrow E$, $Kq'i_0=fq$, $Kq'i_1=gq$. Moreover,
$Kq'$ is a homotopy relative to $\Sigma^{n-1}A$. Hence, if $s_0$ is the base point of $\cn\Sigma^{n-1}A/\Sigma^{n-1}A\cong \Sigma^n A$ and $e_0$ is the base point of $E$, then for $s\in\Sigma^{n-1}A$ and $t\in I$ we have that 
$Kq'(s,t)=K\,\incl(s,\frac{1}{2})=fq(s)=f(s_0)=e_0$.

Then there exists a continuous map $K'':\cyl\Sigma^nA\rightarrow E$ such that $K''\circ\cyl q=Kq'$. Hence, $fq(x)=Kq'i_0(x)=K''(q\times\id_I)i_0(x)=K''(q(x),0)$. Since $q$ is surjective, $K''i_0=f$. In a similar way $K''i_1=g$, and hence $K''$ is a homotopy between $f$ and $g$. Thus, $p_*$ is injective.

We will prove now that $p_*$ is surjective. Let $[f]\in\pi_n^A(B)$. There is a commutative diagram
\begin{displaymath}
\xymatrix{\Sigma^{n-1}A \ar[r]^(.6){c_{e_0}} \ar[d]_\incl & E \ar[d]^p \\ \cn\Sigma^{n-1}A
  \ar[r]_(.62){fq} & B}
\end{displaymath}
where $c_{e_0}$ is the constant map $e_0$. By hypothesis there exists $g:\cn\Sigma^{n-1} A\rightarrow E$ such that $pg=fq$
and $g|_{\Sigma^{n-1}A}=c_{e_0}$. Then, there exists $g':\Sigma^n A\rightarrow E$ such that $g'q=g$. Hence, $pg'q=pg=fq$ and since $q$ is surjective we obtain that $pg'=f$. Thus, $p_*$ is surjective.

It remains to prove that $p_\ast:\pi^A_0(E)\to\pi^A_0(B)$ is an isomorphism. Note that $\pi^A_0(E)$ and $\pi^A_0(B)$ are groups since $A$ is an $H$-cogroup. This will be used to prove that $p_\ast$ is a monomorphism as we will show that $pf\simeq \ast$ implies $f\simeq\ast$.

Since $p$ has the RLP with respect to $\ast\to A$, it follows that $p_\ast$ is an epimorphism. Suppose now that $f:A\to E$ is a map such that $pf\simeq \ast$. Hence, there exists a map $g:\cn A\to B$ such that $g|_A=pf$.
\begin{displaymath}
\xymatrix{A \ar[r]^{f} \ar[d]_\incl & E \ar[d]^p \\ \cn A
  \ar[r]_(.53){g} & B}
\end{displaymath}
By hypothesis there exists a map $g':\cn A\to E$ such that $pg'=g$ and $g'|A=f$. Hence, $f$ can be extended to $\cn A$ and then $f\simeq \ast$.

(b) $\Rightarrow$ (c) follows by standard arguments.
\end{proof}

From the previous proposition we obtain some interesting corollaries.

\begin{coro} \label{coro_I-inj=J-inj_WE}
If $A$ is an $H$-cogroup then $\mathcal{I}_A\textnormal{-inj}=\mathcal{J}_A\textnormal{-inj}\cap \textnormal{WE}$.
\end{coro}

\begin{coro}
If $A$ is an $H$-cogroup then every generalized CW($A$)-complex is an $A$-cofibrant space.
\end{coro}

\begin{coro}
If $A$ is an $H$-cogroup then the inclusion maps $\ast \to A$, $\Sigma^{n-1}A\hookrightarrow \cn\Sigma^{n-1}A$, $n\in\mathbb{N}$, and $i_0: \cyl^{n-1} A\to \cyl^n A$, $n\in\mathbb{N}$, are $A$-cofibrations.
\end{coro}

\begin{rem} \label{rem_case_S0}
Proposition \ref{propfibyed} also holds if $A=S^0$ and $E$ is path-connected. Note that the hypothesis of $A$ being an $H$-cogroup is only used to prove injectivity of the map $p_\ast:\pi^A_0(E)\to \pi^A_0(B)$ which follows trivially if $A=S^0$ and $E$ is path-connected.

However, this does not hold if the space $E$ is not path-connected as the following example shows. Let $E=\{a,b,c\}$ with the discrete topology and with base point $a$, let $B=S^0$ and let $p:E\to B$ be defined by $p(a)=1$, $p(b)=-1$, $p(c)=-1$. It is easy to verify that $p$ satisfies part (b) of proposition \ref{propfibyed} but it does not satisfy part (a) since it does not induce an isomorphism $p_\ast:\pi^{S^0}_0(E)\to \pi^{S^0}_0(B)$.
\end{rem}

If $X$ is a topological space, $X^I$ will denote the space of continuous maps from $I$ to $X$ which do not necessarily preserve base points. We give $X^I$ the compact-open topology. If $x_0$ is the base point of $X$, the constant map from $I$ to $X$ with value $x_0$ will be the base point of $X^I$. We define the maps $\ev_0^X,\ev_1^X:X^I\to X$ by $\ev_0^X(\alpha)=\alpha(0)$ and $\ev_1^X(\alpha)=\alpha(1)$. Also, if $f:X\to Y$ is a continuous map, we define $f^I:X^I\to Y^I$ by $f^I(\alpha)=f\circ \alpha$.

If $f:\cyl X\to Y$ is a continuous map, $f^\sharp$ will denote the continuous map from $X$ to $Y^I$ defined by the exponential law. Also, if $g:X\to Y^I$ is a continuous map, $g^\flat$ will denote the continuous map from $\cyl X$ to $Y$ defined by the exponential law.

\begin{lemma} \label{lemma_trivial_A-fib}
Let $X$ and $Y$ be topological spaces, let $p:X\rightarrow Y$ be an $A$-fibration and let $X\underset{Y}{\times}Y^I$ be the pullback
\begin{displaymath}
\xymatrix{X\underset{Y}{\times}Y^I \ar[d]_{\pr_2} \ar[r]^(.55){\pr_1}
   \ar@{}[dr]|{pull} & X \ar[d]^{p} \\ Y^I \ar[r]_{\ev_0^Y} & Y }
\end{displaymath}
If $A$ is an H-cogroup, the map $(\ev_0^X,p^I):X^I\rightarrow X\underset{Y}{\times}Y^I$ is a trivial $A$-fibration.
\end{lemma}

\begin{proof}
Suppose that there is a commutative diagram
\begin{displaymath}
\xymatrix{ \Sigma^{n-1}A \ar[r]^f \ar[d]_{\incl} & X^I \ar[d]^{(\ev_0^X,p^I)} \\
  \cn\Sigma^{n-1}A \ar[r]_{(g_1,g_2)} &  X\underset{Y}{\times}Y^I  }
\end{displaymath}
with $n\in\N$. Then, there is a commutative diagram
\begin{displaymath}
\xymatrix@C=35pt{\cyl(\Sigma^{n-1}A)\cup i_0(\cn\Sigma^{n-1}A)
  \ar[r]^(.75){f^\flat \cup g_1} \ar[d]_{\incl} & X \ar[d]^{p} \\
  \cyl(\cn\Sigma^{n-1}A) \ar[r]_(.6){g_2^\flat} &  Y  }
\end{displaymath}
Since $p$ is an $A$-fibration, by item (a) of \ref{lemma_homeo_pairs_1} there exists a map $H':\cyl\cn\Sigma^{n-1}A\rightarrow X$ such that
$pH'=g_2^\flat$ and $H'\incl=f^\flat\cup g_1$. Thus, by the exponential law we obtain a continuous map $H:\cn\Sigma^{n-1}A \rightarrow X^I$ and
$H$ is the desired lift.

Suppose now that there is a commutative diagram
\begin{displaymath}
\xymatrix{ \cyl^{n-1}A \ar[r]^(.55){f} \ar[d]_{i_0} & X^I \ar[d]^{(\ev_0^X,p^I)} \\
  \cyl^{n}A \ar[r]_(0.42){(g_1,g_2)} &  X\underset{Y}{\times}Y^I  }
\end{displaymath}
with $n\in\N$. Applying the exponential law we obtain a commutative diagram
\begin{displaymath}
\xymatrix{ \cyl(i_0(\cyl^{n-1}A)) \cup i_0(\cyl(\cyl^{n-1}A)) \ar[r]^(0.8){f^\flat\cup g_1} \ar[d]_{\incl} & X \ar[d]^{p} \\
  \cyl\,\cyl^{n}A \ar[r]_(.53){g_2^\flat} &  Y  }
\end{displaymath}
Since $p$ is an $A$-fibration, by item (c) of lemma \ref{lemma_homeo_pairs_1} there exists a map $H':\cyl\,\cyl^{n}A\to X$ such that $H'\incl=f^\flat\cup g_1$ and $pH'=g_2^\flat$. Now, from $H'$ and the exponential law we get the desired lift $H:\cyl^{n}A\to X^I$.

The right lifting property with respect to $\ast\to A$ can be proved in a similar way.

Thus, by \ref{propfibyed}, the map $(\ev_0^X,p^I)$ is a trivial $A$-fibration.
\end{proof}

\begin{definition}
Let $p:X\to Y$ be a continuous map between pointed topological spaces. We say that $p$ is a pointed Hurewicz fibration if $p$ is a pointed map which has the homotopy lifting property (in the category of pointed topological spaces) with respect to any pointed space.
\end{definition}

If $Z$ is a topological space we define the map $s_Z:Z\to Z^I$ by $s_Z(z)(t)=z$ for $z\in Z$ and $t\in I$.

\begin{prop} \label{prop_trivial_cofibs}
Let $B$ and $Z$ be topological spaces and let $i:B\rightarrow Z$ be a continuous map. If $A$ is an H-cogroup, the following are equivalent:
\begin{enumerate}
\item[(a)] $i \in \textnormal{Cof} \cap \textnormal{WE}$.
\item[(b)] $i$ has the LLP with respect to $\textnormal{Fib}$.
\item[(c)] $i \in \textnormal{Cof}$ and $i$ is a strong deformation retract.
\end{enumerate}
\end{prop}

\begin{proof}
(a)$\Rightarrow$(c) We consider the usual factorization $i=pj$ where $p$ is a Hurewicz fibration and $j$ is a strong deformation retract. It is easy to verify that $p$ is a pointed Hurewicz fibration and hence $p\in\textnormal{Fib}$. On the other hand, $j$ is an $A$-weak equivalence, and since $i$ is also an $A$-weak equivalence it follows that $p\in\textnormal{WE}$. Hence, $p\in \textnormal{Fib}\cap \textnormal{WE}$. Since $i\in\mathrm{Cof}$, there exists a lift $u$ such that the following diagram commutes.
\begin{displaymath}
\xymatrix{B \ar[r]^(0.4){j} \ar[d]_i & B\underset{Z}{\times}Z^I \ar[d]^{p} \\
  Z \ar[r]_{\id_Z} \ar[ru]^u & Z}
\end{displaymath}
Thus, $i$ is a retract of $j$ and it follows that $i$ is also a strong deformation retract.

(c)$\Rightarrow$(a) Follows immediately.

(b)$\Rightarrow$(c) Clearly $i\in\mathrm{Cof}$. Also, since the map $B\to \ast$ is a pointed Hurewicz fibration, it is an $A$-fibration. Hence, there exists a map $r$ such that the following diagram commutes:
\begin{displaymath}
\xymatrix{B \ar[d]_{i} \ar[r]^{\id_B} & B \ar[d] \\ Z \ar[r]
  \ar@{.>}[ru]^r & \ast}
\end{displaymath}
Then $ri=\id_B$.

Also, it is easy to prove that the map $(\ev_0^Z,\ev_1^Z):Z^I\rightarrow Z\times Z$ is a pointed Hurewicz fibration. Hence it is an $A$-fibration.

By hypothesis, there exists a map $H$ such that the following diagram commutes:
\begin{displaymath}
\xymatrix{B \ar[d]_{i} \ar[r]^(0.45){s_Z\circ i} & Z^I \ar[d]^{(\ev_0^Z,\ev_1^Z)}\\ Z
  \ar[r]_(0.4){(ir,\id_Z)} \ar@{.>}[ru]^{H} & Z\times Z}
\end{displaymath}
Then, $H^\flat:ir\simeq\id_Z\ \mathrm{rel}(i(B))$. Thus, $i$ is a strong deformation retract.

(c)$\Rightarrow$(b) Let $p:X\rightarrow Y$ be an $A$-fibration and
suppose there is a commutative diagram
\begin{displaymath}
\xymatrix{B \ar[d]_{i} \ar[r]^{\alpha} & X \ar[d]^p \\ Z
  \ar[r]_{\beta} & Y}
\end{displaymath}
Since $i$ is a strong deformation retract, there exist
$r:Z\rightarrow B$ and $h:IZ\rightarrow Z$ such that $ri=\id_B$ and
$h:ir\simeq\id_Z\ \mathrm{rel}(i(B))$. By \ref{lemma_trivial_A-fib},
the map $(\ev_0^X,p^I):X^I\rightarrow X\underset{Y}{\times}Y^I$ is a trivial $A$-fibration and since $i\in\mathrm{Cof}$ and
\begin{eqnarray}
(\ev_0^X,p^I)s_X\alpha & = & (\ev_0^Xs_X\alpha,p^Is_X\alpha)= (\alpha,s_Yp\alpha)=
(\alpha,s_Y\beta i)=(\alpha,\beta^Is^Zi) = \nonumber \\ & = & (\alpha ri,\beta^Ih^\sharp i) = (\alpha r,\beta^Ih^\sharp)i \nonumber
\end{eqnarray}
there exists a continuous map $H$ that such that the following diagram commutes
\begin{displaymath}
\xymatrix@C=30pt{B \ar[d]_{i} \ar[r]^{s_X\alpha} & X^I
  \ar[d]^{(\ev_0^X,p^I)} \\ Z \ar[r]_(0.4){(\alpha r,\beta^Ih^\sharp)} \ar@{.>}[ru]^H
  & X\underset{Y}{\times}Y^I}
\end{displaymath}
Let $u=\ev_1^XH$. Then
$$pu=p\,\ev_1^XH=\ev_1^Yp^IH=\ev_1^Y\beta^Ih^\sharp=\beta \ev_1^Zh^\sharp=\beta h i_1=\beta$$
and
$$ui=\ev_1^XHi=\ev_1^Xs_X\alpha=\alpha.$$
Then, $u$ is the desired lift. Thus $i$ has the LLP with respect to $\textnormal{Fib}$.
\end{proof}

\begin{coro} \label{coro_J-cof=I-cof_WE}
If $A$ is an $H$-cogroup then $\mathcal{J}_A\textnormal{-cof}=\mathcal{I}_A\textnormal{-cof}\cap \textnormal{WE}$.
\end{coro}
\begin{proof}
Note that $\mathcal{J}_A\textnormal{-cof}$ is the class of maps that have the LLP with respect to the class $\textnormal{Fib}$ since $\textnormal{Fib}=\mathcal{J}_A\textnormal{-inj}$. On the other hand, $\mathcal{I}_A\textnormal{-cof}$ is the class of maps that have the LLP with respect to the class $\textnormal{Fib}\cap \textnormal{WE}$ since $\mathcal{I}_A\textnormal{-inj}=\mathcal{J}_A\textnormal{-inj}\cap \textnormal{WE}=\textnormal{Fib}\cap \textnormal{WE}$ by \ref{coro_I-inj=J-inj_WE}. Hence, $\mathcal{I}_A\textnormal{-cof}=\textnormal{Cof}$. The result follows from \ref{prop_trivial_cofibs}.
\end{proof}

\begin{theo} \label{theo_model_category}
Let $A$ be a compact and T$_1$ topological $H$-cogroup. Then the category of topological spaces with $\textnormal{Fib}$, $\textnormal{Cof}$ and $\textnormal{WE}$ as the classes of fibrations, cofibrations and weak equivalences respectively is a cofibrantly generated model category.
\end{theo}

\begin{proof}
Note that by \ref{propfibyed} a map belongs to $\textnormal{Fib}\cap\textnormal{WE}$ if and only if it has the RLP with respect to the class $\mathcal{I}_A$. Also, by definition, a map is a fibration if and only if it has the RLP with respect to the class $\mathcal{J}_A$. Hence, by theorem 11.3.1 of \cite{Hir}, it suffices to prove that
\begin{enumerate}
\item[(0)] $\textnormal{WE}$ is closed under retracts and satisfies the `two out of three' axiom.
\item Both $\mathcal{I}_A$ and $\mathcal{J}_A$ permit the small object argument.
\item $\mathcal{J}_A\textnormal{-cof}\subseteq\mathcal{I}_A\textnormal{-cof}\cap \textnormal{WE}$.
\item $\mathcal{I}_A\textnormal{-inj}\subseteq\mathcal{J}_A\textnormal{-inj}\cap \textnormal{WE}$.
\item $\mathcal{I}_A\textnormal{-cof}\cap \textnormal{WE}\subseteq \mathcal{J}_A\textnormal{-cof}$ or $\mathcal{J}_A\textnormal{-inj}\cap \textnormal{WE} \subseteq \mathcal{I}_A\textnormal{-inj}$.
\end{enumerate}

The proof of (0) is straightforward.

To prove (1) note that the maps in $\mathcal{I}_A$ or $\mathcal{J}_A$ are closed inclusions into T$_1$ spaces since $A$ is a T$_1$ space. Also, the domains of those maps are compact spaces since $A$ is a compact space. Since compact topological spaces are finite relative to closed T$_1$ inclusions (\cite{Hov}, proposition 2.4.2) we obtain that $\mathcal{I}_A$ and $\mathcal{J}_A$ permit the small object argument.

Items (2), (3) and (4) follow from \ref{coro_I-inj=J-inj_WE} and \ref{coro_J-cof=I-cof_WE}.
\end{proof}

\medskip

\begin{rem}
Note that this model category structure is finitely generated. Note also that, with this model category structure, every topological space is fibrant and every generalized CW($A$)-complex is cofibrant.
\end{rem}

\begin{rem} \label{rem_top_S0}
Theorem \ref{theo_model_category} also holds for $A=S^0$ if we work in the category of path-connected topological spaces since proposition \ref{propfibyed} also holds under these hypothesis (cf. remark \ref{rem_case_S0}). More precisely, the category of pointed and path-connected topological spaces with the classes of $S^0$-fibrations, $S^0$-cofibrations and $S^0$-weak equivalences as the classes of fibrations, cofibrations and weak equivalences respectively is a cofibrantly generated model category.

Moreover, this model category structure coincides with the usual model category structure in the category of pointed topological spaces restricted to the full subcategory of path-connected spaces.
\end{rem}

As a consequence of the small object argument (\cite{Hir}, proposition 10.5.16) we obtain the following CW($A$)-approximation theorem.

\begin{theo} \label{theo_CWA-aproximation}
Let $A$ be a compact and T$_1$ topological H-cogroup. Let $X$ be a topological space. Then there exists a generalized CW($A$)-complex $Z$ together with an $A$-weak equivalence $f:Z\to X$.
\end{theo}

\begin{proof}
Since the class $\mathcal{I}_A$ permits the small object argument we can factorize the map $\ast \to X$ into a relative $\mathcal{I}_A$-cell complex $i:\ast\to W$ followed by an $\mathcal{I}_A$-injective $p:W\to X$ (note that $p$ is an $A$-weak equivalence by \ref{coro_I-inj=J-inj_WE}). Moreover, in this factorization we only need to take a countably infinite sequence of pushouts since the domains of the maps of $\mathcal{I}_A$ are finite relative to $\mathcal{I}_A$ as we pointed out in the previous proof.

Thus, the space $W$ is constructed in much the same way as a generalized CW($A$)-complex, but some `extra' cells are attached taking pushouts with maps $\cyl^{k-1}A\to \cyl^{k}A$. However, since the inclusion $\cyl^{k-1}A\hookrightarrow \cyl^{k}A$ is a strong deformation retract, we can replace each step $W_n$ of the construction of $W$ by a generalized CW($A$)-complex $Z_n$ constructed from $Z_{n-1}$ by attaching the same cells that are attached to $W_{n-1}$, excluding those coming from maps $\cyl^{k-1}A\to \cyl^{k}A$. By a similar reasoning as the one in the proof of theorem 4.1 of \cite{MO1} we obtain that $Z_n$ is a strong deformation retract of $W_n$ and if we define $Z$ as the colimit of the spaces $Z_n$, $n\in\N$, we get that $Z$ is a generalized CW($A$)-complex and a strong deformation retract of $W$. Composing the strong deformation retract $j:Z\to W$ with the map $p$ we get an $A$-weak equivalence $f:Z\to X$.
\end{proof}

Now, we will apply the model category structure defined above to obtain an $A$-based version of Whitehead's theorem. To this end, we need the following proposition which relates Quillen's left homotopies with usual homotopies.

\begin{prop}
Let $A$ be a compact and T$_1$ topological $H$-cogroup. Let $X$ be an $A$-cofibrant space and let $Y$ be another topological space. Let $g_1,g_2:X\to Y$ be continuous maps. Then $g_1$ and $g_2$ are homotopic (in the usual sense) if and only if they are left homotopic (in the sense of Quillen).
\end{prop}

\begin{proof}
As usual, if $g_1$ and $g_2$ are homotopic, factorizing the map $i_0+i_1:X\lor X \to \cyl X$ into an $A$-cofibration followed by an $A$-weak equivalence we obtain a cylinder object $X'$ and a left homotopy $H:X'\to Y$ between $g_1$ and $g_2$.

Suppose now that $g_1$ and $g_2$ are left homotopic. Note that $Y^I$ is a path object for $Y$ since the map $(\ev_0,\ev_1):Y^I\to Y\times Y$ is a pointed Hurewicz fibration and the map $s_Y:Y\to Y^I$ (which was defined by $s_Y(y)(t)=y$ for $y\in Y$ and $t\in I$) is a homotopy equivalence. Since $X$ is $A$-cofibrant and $Y^I$ is a path object for $Y$ then by proposition 1.2.5 of \cite{Hov} there exists a right homotopy $K:X\to Y^I$ between $g_1$ and $g_2$. Applying the exponential law we get that $g_1$ and $g_2$ are homotopic in the usual sense.
\end{proof}

The following result is the $A$-based version of Whitehead's theorem mentioned before. It generalizes theorem 5.10 of \cite{MO1} and is related to theorem E.1 of chapter 2 of \cite{Far}.

\begin{theo}
Let $A$ be a compact and T$_1$ topological $H$-cogroup. Let $X$ and $Y$ be topological spaces and let $f:X\to Y$ be an $A$-weak equivalence. If $X$ and $Y$ are $A$-cofibrant spaces then $f$ is a homotopy equivalence. In particular, this holds if $X$ and $Y$ are generalized CW($A$)-complexes.
\end{theo}

\begin{proof}
The result follows by standard arguments applying the previous proposition and the fact that every space is fibrant in this model category structure (cf. proposition 1.2.8 of \cite{Hov}).
\end{proof}

\section{Quillen functors}

In this section we will give three Quillen adjunctions related to the $A$-based cofibrantly generated model category structure developed before. The first of them shows that, under certain hypotheses, the exponential law is a Quillen adjunction. The second and third ones relate the model category structure defined in the previous section for different choices of the base space $A$. Namely, we give adjunctions between an $A$-based and a $B$-based structure when the space $B$ is a retract of $A$ and when the space $B$ is an iterated suspension of $A$.

From now on, if $A$ is a compact and T$_1$ topological $H$-cogroup, $\top_A$ will denote the model category of pointed topological spaces with the $A$-based cofibrantly generated model category structure defined in the previous section.

\begin{lemma}
Let $E$ and $B$ be topological spaces and let $p:E\to B$ be a continuous map. If $p$ is an $S^0$-fibration then $p$ is an $S^1$-fibration.
\end{lemma}

\begin{proof}
We have to prove that the map $p$ has the homotopy lifting property with respect to $\cyl^n S^1$ for all $n\in\N_0$. Let $n\in\N_0$ and let $f:\cyl^n S^1 \to E$ and $H:\cyl(\cyl^n S^1)\to B$ be continuous maps such that $Hi_0=pf$. For $j\in\N$ let $q_j:I^{j}\to I^{j}/(\{0,1\}\times I^{j-1})$ be the quotient map and let $\varphi_j:I^{j}/(\{0,1\}\times I^{j-1})\to \cyl^{j-1}S^1$ be the homeomorphism defined by $\varphi[(t_1,t_2,\ldots,t_j)]=[(e^{2\pi t_1 i},t_2,\ldots,t_j)]$.

There is a commutative diagram of pointed spaces and maps
\begin{displaymath}
\xymatrix{I^{n+1} \ar[r]^(0.32){q_{n+1}} \ar[d]_{i_0} & I^{n+1}/(\{0,1\}\times I^{n}) \ar[r]^(0.67){\varphi_{n+1}} \ar[d]^{\overline{i_0}} & \cyl^{n} S^1 \ar[r]^(0.57){f} \ar[d]^{i_0} & E \ar[d]^{p} \\
I^{n+2} \ar[r]^(0.29){q_{n+2}} & I^{n+2}/(\{0,1\}\times I^{n+1}) \ar[r]^(0.65){\varphi_{n+2}} & \cyl(\cyl^{n} S^1) \ar[r]^(0.62)H & B}
\end{displaymath}
where $\overline{i_0}$ is the map that makes commutative the left square. 

Let $e_0$ be the base point of $E$. Consider the following diagram in the category of (unbased) topological spaces.
\begin{displaymath}
\xymatrix{i_0(I^{n+1})\cup (\{0,1\}\times I^{n+1}) \ar[r]^(0.78){\alpha} \ar[d]_\incl & E \ar[d]^{\mathcal{O}(p)} \\
I^{n+2} \ar[r]_{\mathcal{O}(H\varphi_{n+2}q_{n+2})} & B}
\end{displaymath}
where $\alpha$ is defined by $\alpha(x,0)=f\varphi_{n+1}q_{n+1}(x)$ for all $x\in I^{n+1}$ and $\alpha(t,y)=e_0$ for all $t\in\{0,1\}$ and $y\in I^{n+1}$. Note that $\alpha$ is well defined and continuous. Note also that the previous diagram is commutative.

Since $p$ is an $S^0$-fibration then, by \ref{prop_S0_fib}, $\mathcal{O}(p)$ is a Serre fibration. Hence, there exists a continuous map $H':I^{n+2}\to E$ such that $H'\incl=\alpha$ and $\mathcal{O}(p)H'=\mathcal{O}(H\varphi_{n+2}q_{n+2})$. Thus, $H'(t,y)=e_0$ for all $t\in\{0,1\}$ and $y\in I^{n+1}$. Then, there exists a continuous map $\overline{H'}:I^{n+2}/(\{0,1\}\times I^{n+1})\to E$ such that $\overline{H'}q_{n+2}=H'$.

Since $\overline{H'}$ preserves base points it can be seen as a map in the category of pointed topological spaces. Thus, let $K=\overline{H'}\circ(\varphi_{n+2})^{-1}:\cyl(\cyl^{n} S^1)\to E$. It is easy to prove that $Ki_0=f$ and $pK=H$.
\end{proof}

\begin{prop}
Let $A$ be a compact and Hausdorff topological $H$-cogroup. Let $L:\top_{S^1} \to \top_A$ be the functor $\_\_\land A$ and let $R:\top_A \to \top_{S^1}$ be the functor $\hom(A,\_\_)$. Let $\varphi$ be the natural isomorphism expressing $R$ as a right adjoint of $L$. Then $(L,R,\varphi)$ is a Quillen adjunction.
\end{prop}

\begin{proof}
It suffices to prove that $R$ is a right Quillen functor. First, we will prove that $R$ preserves fibrations. Let $p:X\to Y$ be an $A$-fibration. Let $n\in\N$ and suppose we have a commutative diagram
\begin{displaymath}
\xymatrix@C=30pt{\cn\Sigma^{n-1}S^0 \ar[d]_{i_0} \ar[r]^(0.45){f} & \hom(A,X)
  \ar[d]^{R(p)} \\ \cyl \cn\Sigma^{n-1}S^0 \ar[r]_(0.45){H} & \hom(A,Y) }
\end{displaymath}
Then, by naturality of $\varphi$, there is a commutative diagram
\begin{displaymath}
\xymatrix@C=30pt{\cn\Sigma^{n-1}S^0\sm A \ar[d]_{L(i_0)} \ar[r]^(0.65){\varphi^{-1}(f)} & X
  \ar[d]^{p} \\ (\cyl \cn\Sigma^{n-1}S^0 ) \sm A \ar[r]_(0.7){\varphi^{-1}(H)} & Y }
\end{displaymath}
Since
$$(\cn\Sigma^{n-1}S^0)\sm A = I \sm S^{n-1} \sm S^0 \sm A = \cn\Sigma^{n-1} A$$
and 
$$(\cyl\cn\Sigma^{n-1}S^0)\sm A = I_+ \sm I \sm S^{n-1} \sm S^0 \sm A = \cyl \cn\Sigma^{n-1} A$$
and $p$ is an $A$-fibration, there exists a continuous map $\widetilde{H}:(\cyl \cn\Sigma^{n-1}S^0 ) \sm A \to X$ such that $\widetilde{H} \circ L(i_0) = \varphi^{-1}(f)$ and $p \circ \widetilde{H} = \varphi^{-1}(H)$. The map $\varphi(\widetilde{H}):\cyl \cn\Sigma^{n-1}S^0 \to \hom(A,X)$ is the desired lift.

In a similar way one can prove that the map $R(p)$ has the HLP with respect to $S^0$ and with respect to $\cyl^n S^0$ for all $n\in\N$. Hence $R(p)$ is an $S^0$-fibration. Thus, by the previous lemma, $R(p)$ is an $S^1$-fibration.

We will prove now that $R$ preserves weak equivalences. Let $g:X\to Y$ be a continuous map and let $n\in\N_0$. By the exponential law, there is a commutative diagram
\begin{displaymath}
\xymatrix@C=30pt{[S^n\sm A,X] \ar[d]_{g_\ast} \ar[r]^(0.45){\overline{\varphi}} & [S^n,\hom(A,X)] 
  \ar[d]^{R(g)_\ast} \\ [S^n\sm A,Y] \ar[r]_(0.45){\overline{\varphi}} & [S^n,\hom(A,Y)] }
\end{displaymath}
where $\overline{\varphi}$ is the natural isomorphism induced by $\varphi$. Since $S^n\sm A=\Sigma^{n}A$, from the commutativity of the previous diagram it follows that if $g$ is an $A$-weak equivalence then $R(g):\hom(A,X)\to \hom(A,Y)$ is an $S^0$-weak equivalence, and hence it is an $S^1$-weak equivalence.

Thus, $R$ is a right Quillen functor.
\end{proof}

\begin{rem}
Although the functor $R$ of the previous proposition takes $A$-weak equivalences to $S^0$-weak equivalences and $A$-fibrations to $S^0$-fibrations it can not be regarded as a Quillen functor to $\top_{S^0}$ since $\top_{S^0}$ is not a Quillen model category with our structure (cf. remark \ref{rem_top_S0}).
\end{rem}

Now we turn our attention to the second adjunction mentioned at the beginning of this section which is essentially given by the following proposition.

\begin{prop}
Let $A$ and $B$ be (pointed) topological spaces and let $f$ be a continuous map.
\begin{enumerate}
\item If $B$ is a retract of $A$ and $f$ is an $A$-fibration then $f$ is a $B$-fibration.
\item If there exist continuous maps $\alpha:A\to B$ and $\beta:B\to A$ such that $\alpha\circ\beta\simeq \id_B$ and $f$ is an $A$-weak equivalence then $f$ is a $B$-weak equivalence.
\end{enumerate}
\end{prop}

\begin{proof}
($1$) If $B$ is a retract of $A$ then the maps $i_0:B\to \cyl B$, $i_0:\cn \Sigma^{n-1} B\to \cyl \cn\Sigma^{n-1}B$, $n\in\N$, and $i_0:\cyl^{n-1} B\to \cyl^n B$, $n\in\N$, are retracts of the maps $i_0:A\to \cyl A$, $i_0:\cn \Sigma^{n-1} A\to \cyl \cn\Sigma^{n-1}A$, $n\in\N$, and $i_0:\cyl^{n-1} A\to \cyl^n A$, $n\in\N$, respectively. If $f$ is an $A$-fibration, $f$ has the RLP with respect to this last collection of maps and hence $f$ has the RLP with respect to the first collection of maps. Thus $f$ is a $B$-fibration.

($2$) For each $n\in\N_0$ there is a commutative diagram
\begin{displaymath}
\xymatrix@C=30pt{[\Sigma^{n}B,X] \ar[d]_{f_\ast} \ar[r]^(0.5){(\Sigma^n\alpha)^\ast} & [\Sigma^{n} A,X] \ar[d]_{f_\ast} \ar[r]^(0.5){(\Sigma^n\beta)^\ast} & [\Sigma^{n} B,X]  \ar[d]^{f_\ast} \\  
[\Sigma^{n}B,Y] \ar[r]^(0.5){(\Sigma^n\alpha)^\ast} & [\Sigma^{n}A,Y] \ar[r]^(0.5){(\Sigma^n\beta)^\ast} & [\Sigma^{n}B,Y] }
\end{displaymath}
Since the horizontal compositions are the identity maps, the map $f_\ast:[\Sigma^{n}B,X]\to[\Sigma^{n}B,Y]$ is a retract of the map $f_\ast:[\Sigma^{n}A,X]\to[\Sigma^{n}A,Y]$ which is an isomorphism by hypothesis. Then $f_\ast:[\Sigma^{n}B,X]\to[\Sigma^{n}B,Y]$ is an isomorphism for all $n\in\N_0$ and hence $f$ is a $B$-weak equivalence.
\end{proof}

\begin{prop}
Let $A$ and $B$ be compact and T$_1$ topological $H$-cogroups. Let $L:\top_B \to \top_A$ and $R:\top_A \to \top_B$ be the identity functors. Let $\phi$ be the natural isomorphism defined by the identity maps in all objects. If $B$ is a retract of $A$ then $(L,R,\phi)$ is a Quillen adjunction.
\end{prop}

\begin{proof}
It suffices to prove that $R$ is a right Quillen functor, which follows from the previous proposition.
\end{proof}

Clearly, if $B$ is a retract of $A$ and $A$ and $B$ are homotopy equivalent then the adjunction of above turns out to be a Quillen equivalence.

Now we turn our attention to the third adjunction which deals with suspensions.

\begin{prop}
Let $A$ be a (pointed) topological space, let $n\in\N$ and let $f$ be a continuous map. If $A$ is the suspension of a locally compact and Hausdorff space and $f$ is an $A$-fibration then $f$ is a $\Sigma^n A$-fibration.
\end{prop}

\begin{proof}
It is clear that for all $m\in \N$, $f$ has the homotopy lifting property with respect to the spaces $\cn \Sigma^{m-1}(\Sigma^n A)$ since $f$ is an $A$-fibration. It remains to prove that $f$ has the homotopy lifting property with respect to the spaces $\cyl^m\Sigma^n A$ for all $m\in\N_0$.

Let $m\in\N_0$. By \ref{IX_is_CW(A)},  $(\cyl(\cyl^m\Sigma^n A),i_0(\cyl^m\Sigma^n A))$ is a relative CW($A$)-complex. Then, by \ref{propfibyed}, the map $i_0:\cyl^m\Sigma^n A\to \cyl(\cyl^m\Sigma^n A)$ has the LLP with respect to the trivial $A$-fibrations and hence it is an $A$-cofibration. But since it is also a homotopy equivalence, from \ref{prop_trivial_cofibs} it follows that it has the LLP with respect to the $A$-fibrations. Hence, $f$ has the homotopy lifting property with respect to the space $\cyl^m\Sigma^n A$.
\end{proof}

Since every $A$-weak equivalence is a $\Sigma^n A$-weak equivalence, from the previous proposition we obtain the following.

\begin{prop}
Let $A$ be the suspension of a compact and Hausdorff space and let $n\in\N$. Let $L:\top_{\Sigma^n A} \to \top_A$ and $R:\top_A \to \top_{\Sigma^n A}$ be the identity functors. Let $\phi$ be the natural isomorphism defined by the identity maps in all objects. Then $(L,R,\phi)$ is a Quillen adjunction.
\end{prop}

If, in addition, $A$ is homotopy equivalent to a compact and $(k-1)$-connected CW-complex of dimension $k$ for some $k\in\N$, then the adjunction of the previous proposition induces an equivalence of categories between $\ho\top_{\Sigma^n A}$ and the full subcategory of $\ho\top_A$ whose objects are the $A$-$(n-1)$-connected spaces. The proof of this fact follows from theorem 3.2 of \cite{EHR} in the following way. If $k\geq 2$, from homology decomposition (cf. theorem 4H.3 of \cite{Hat}) it follows that a $(k-1)$-connected CW-complex of dimension $k$ is homotopy equivalent to a wegde of spheres of dimension $k$. Note that this also holds for $k=1$.

Then, if $n\in\N$ an $A$-cofibrant space is $(n+k-1)$-connected if and only if it is $A$-$(n-1)$-connected. Indeed, the first implication is clear while the converse can be proved in a similar way to the first item of proposition 3.1 of \cite{EHR}.

On the other hand, if $B$ is a wegde of spheres of dimension $k$ and $X$ is a pointed space then $\pi_n^B(X)$ is a direct product of copies of the group $\pi_{n+k}(X)$. Hence a map is a $B$-weak equivalence if and only if it induces isomorphisms in $\pi_m$ for $m\geq k$. Thus, under the assumptions on the space $A$, a map is an $A$-weak equivalence if and only if it induces isomorphisms in $\pi_m$ for $m\geq k$. 

Hence, applying theorem 3.2 of \cite{EHR} one gets an equivalence of categories
\begin{displaymath}
\begin{array}{rcl}
\ho\top_{\Sigma^n A} & \cong & \ho\top|_{(n+k-1)-\textnormal{conn}} \cong  (\ho\top|_{(k-1)-\textnormal{conn}})|_{(n+k-1)-\textnormal{conn}} \cong \\ & \cong & \ho\top_A|_{A-(n-1)-\textnormal{conn}}
\end{array}
\end{displaymath}
where, for a category $\mathscr{C}$ and for $l\in\N_0$, $\mathscr{C}|_{l-\textnormal{conn}}$ and $\mathscr{C}|_{A-l-\textnormal{conn}}$ denote the full subcategories of $\mathscr{C}$ whose objects are the $l$-connected spaces and the $A$-$l$-connected spaces respectively.

Note that although the model category developed in \cite{EHR} is different from ours, their weak $n$-equivalences coincide with our $S^n$-weak equivalences and hence the homotopy category they obtain by inverting the weak $n$-equivalences is the same as $\ho\top_{S^n}$.

\end{document}